\documentclass[reqno, 12pt]{amsart}

\pdfoutput=1

\usepackage[numbers,sort,compress]{natbib}
\usepackage{ytableau}
\usepackage{enumerate}
\usepackage{latexsym}
\usepackage[centertags]{amsmath}
\usepackage{amsfonts}
\usepackage{amssymb}
\usepackage{amsthm}
\usepackage{newlfont}
\usepackage{graphics}
\usepackage{color}
\usepackage{float}
\usepackage{diagbox}
\usepackage{longtable}
\usepackage{rotating}
\usepackage{multirow}
\usepackage[hyphens]{url}
\usepackage{breakurl}
\usepackage{extarrows}
\usepackage[sort,compress,numbers]{natbib}
\usepackage[utf8]{inputenc}

\newtheorem{thm}{Theorem}[section]
\newtheorem{cor}[thm]{Corollary}
\newtheorem{lem}[thm]{Lemma}
\newtheorem{prop}[thm]{Proposition}

\newtheorem{cnj}[thm]{Conjecture}
\newtheorem{dfn}[thm]{Definition}
\newtheorem{rem}[thm]{Remark}
\newtheorem{exa}[thm]{Example}

\allowdisplaybreaks[4]

\title{A parking function interpretation for $(-1)^{k}\nabla m_{2^{k}1^{l}}$}

\author{Menghao Qu$^{1}$ and Guoce Xin$^{2, *}$}
\address{$^{1, 2}$School of Mathematical Sciences,  Capital Normal University, Beijing 100048,  PR China}
\email{$^1$\texttt{menghao.qu@cnu.edu.cn}\ \& $^2$\texttt{guoce\_xin@163.com}}

\begin{document}
\maketitle

\begin{abstract}
Haglund, Morse, and Zabrocki introduced a family of creation operators of Hall-Littlewood polynomials, $\{C_{a}\}$ for any $a\in \mathbb{Z}$, in their compositional refinement of the shuffle (ex-)conjecture. For any $\alpha\vDash n$, the combinatorial formula for $\nabla C_{\alpha}$ is a weighted sum of parking functions. These summations can be converted to a weighted sum of certain LLT polynomials. Thus $\nabla C_{\alpha}$ is Schur positive since Grojnowski and Haiman proved that all LLT polynomials are Schur positive. In this paper, we obtain a recursion that implies the $C$-positivity of $(-1)^{k} m_{2^{k}1^{l}}$, and hence
prove the Schur positivity of $(-1)^{k}\nabla m_{2^{k}1^{l}}$. As a corollary, a parking function interpretation for $(-1)^{k}\nabla m_{2^{k}1^{l}}$
is obtained by using the compositional shuffle theorem of Carlsson and Mellit.
\end{abstract}

\noindent
\begin{small}
 \emph{Mathematic subject classification}: 05A17; 05E05; 05E10.
\end{small}

\noindent
\begin{small}
\end{small}

\section{Introduction}
A symmetric function $f$ is Schur positive if all coefficients in its expansion under the Schur basis $\{s_{\lambda}\}$ are
in the set $\mathbb{N}$ of nonnegative integers. Schur positivity plays a key role in Algebraic Combinatorics. In representation theory, every Schur function corresponds to an irreducible character under the Frobenius characteristic isomorphism.

Modified Macdonald polynomials $\tilde{H}_{\mu}[X;q,t]$ were introduced by Garsia and Haiman in 1993. In \cite{garsia1993graded}, they found a $\mathcal{S}_{n}$-module $\mathcal{H}_{\mu}$ whose bigraded Frobenius characteristic is equal to $\tilde{H}_{\mu}[X;q,t]$. The $\mathcal{H}_{\mu}$ is a $\mathcal{S}_{n}$-submodule of the space of diagonal harmonics, $\mathcal{DH}_{n}$. The bigraded Frobenius characteristic of $\mathcal{DH}_{n}$
was conjectured to be $\nabla e_{n}$, where $\nabla$, which was introduced by Bergeron and Garsia \cite{bergeron1999science}, is a linear operator with
eigenfunctions $\tilde{H}_{\mu}[X;q,t]$. The conjecture was proved by Haiman \cite{haiman2002vanishing} using tools from Algebraic Geometry.
Thus from the perspective of representation theory, $\tilde{H}_{\mu}[X;q,t]$ and $\nabla e_{n}$ are Schur positive. Finding a beautiful combinatorial interpretation for their Schur coefficients is a famous open problem. So far, only a few special cases have been resolved.

Furthermore, the dimension of $\mathcal{DH}_{n}$ is $(n+1)^{n-1}$. This quantity also counts the number of parking functions of size $n$. The combinatorial formula for $\nabla e_{n}$ is known as the shuffle theorem. It was first conjectured by Haglund, Haiman, Loehr, Remmel, and Ulyanov in \cite{haglund2005combinatorial}. In the years since, various special cases and generalizations have been proven and developed by numerous individuals \cite{garsia2002proof,haglund2004proof, garsia2014three, bergeron2016compositional, haglund2018delta,loehr2007square,loehr2008nested,sergel2017proof, blasiak2024llt}. The readers can refer to \cite{van2020shuffle} for a survey on the shuffle theorem. The most important development is the compositional refinement given by Haglund, Morse, and Zabrocki in \cite{haglund2012compositional}.

To maintain brevity in the introduction, we will utilize certain notations that will be explained in detail in Section 2.
For any $f\in \Lambda$ and $a\in \mathbb{Z}$, define
$$C_{a}f[X]=\left(-\frac{1}{q}\right)^{a-1}f\left[X-\frac{1-\frac{1}{q}}{z}\right]\Omega[zX]\Bigg|_{z^a}.$$
For any $\alpha=(\alpha_{1},\alpha_{2},\cdots,\alpha_{l})\vDash n$, let $$C_{\alpha}=C_{\alpha_{1}}C_{\alpha_{2}}\cdots C_{\alpha_{l}}(1),$$
then we have
\begin{thm}\cite[Compositional Shuffle Theorem]{carlsson2018proof}
\begin{equation}\label{Compositional shuffle theorem}
\nabla C_{\alpha}=\sum\limits_{\substack{P\in \mathcal{PF}_{n}\\touch(P)=\alpha}}q^{dinv(P)}t^{area(P)}F_{ides(P)}.
\end{equation}
\end{thm}

For any $\alpha\vDash n$, the combinatorial formula for $\nabla C_{\alpha}$ is a weighted sum of parking functions. By Remark 6.5 in \cite{haglund2008q}, these summations can be converted to a weighted sum of certain LLT polynomials introduced by Lascoux, Leclerc, and Thibon \cite{lascoux1997ribbon} in 1997.
Thus $\nabla C_{\alpha}$ is Schur positive according to a result of Grojnowski and Haiman in \cite{grojnowski2007affine}:
all LLT polynomials are Schur positive. The $\{C_{\alpha}\}$ can be seen as building blocks of many symmetric functions. In \cite{haglund2012compositional}, the authors proved that $e_{n}=\sum_{\alpha\vDash n}C_{\alpha}$. Summing (\ref{Compositional shuffle theorem}) over all $\alpha\vDash n$ gives
\begin{cor}\cite[Shuffle Theorem]{carlsson2018proof}
$$\nabla e_{n}=\sum_{P\in \mathcal{PF}_{n}}q^{dinv(P)}t^{area(P)}F_{ides(P)}.$$
\end{cor}

If a symmetric function $f$ admits a positive $C$-expansion, then we can obtain the Schur positivity of $\nabla f$ and give its parking function interpretation
using results of $\nabla C_{\alpha}$. Based on this idea, Sergel proved the following identity and obtained a parking function interpretation for $(-1)^{n-1}\nabla p_{n}$ in her PhD thesis.
\begin{thm}\cite[Theorem 4.1.4]{sergel2016combinatorics}
$$(-1)^{n-1}p_{n}=\sum\limits_{\alpha\vDash n}[\alpha_{1}]_{q}C_{\alpha}=\sum\limits_{\alpha\vDash n}(1+q+\cdots +q^{\alpha_{1}-1})C_{\alpha}.$$
\end{thm}
\begin{cor}\cite[Theorem 4.1.1]{sergel2016combinatorics}
$$(-1)^{n-1}\nabla p_{n}=\sum\limits_{P\in \mathcal{PF}_{n}}[ret(P)]_{q}q^{dinv(P)}t^{area(P)}F_{ides(P)}.$$
\end{cor}

Our study concerns the following conjecture.
\begin{cnj}\cite[Conjecture IV]{bergeron1999identities}\label{nabla m}
For any partition $\lambda$ and $\mu$,
$$\left<(-1)^{|\mu|-l(\mu)}\nabla m_{\mu},s_{\lambda}\right> \in \mathbb{N}[q,t].$$
\end{cnj}
In two special cases, $\mu=(1^{n})$ and $\mu=(n)$, their Schur positivity and parking function interpretation were already known from the previous conclusion, as $m_{1^{n}}=e_{n}$ and $m_{n}=p_{n}$. Another special case that is already known is the hook case proved by Sergel.

Based on a fact that $C_{\alpha}\big|_{q=1}=(-1)^{|\alpha|-l(\alpha)}h_{\alpha}$ for any composition $\alpha$ and the property of the combinatorial object, bi-brick permutation, of transition matrix between monomial basis and homogeneous basis, Sergel introduced a combinatorial model to find positive $C$-expansion.

\begin{cnj}\cite[Conjecture 3.1]{sergel}
There is a non-negative interger $stat(\Pi)$ so that
\begin{equation*}
(-1)^{|\mu|-l(\mu)}m_{\mu}=\sum_{\mu(\Pi)=\mu}q^{stat(\Pi)}C_{\alpha(\Pi)},    
\end{equation*}
where $\mu(\Pi)$ and $\alpha(\Pi)$ are certain statistics of bi-brick permutation $\Pi$.
\end{cnj}

One can refer to \cite{sergel} for more details as this part of content is not relevant to this article. Sergel gave an explicit formula in the hook case.

\begin{thm}\cite[Theorem 2.1]{sergelparking}
Let $2\leq n\in \mathbb{N}$ and $1\leq k\in \mathbb{N}$, then
$$(-1)^{n-1}m_{n,1^{k}}=\sum\limits_{a=1}^{n}\left(k+1+\sum\limits_{i=1}^{a-1}q^{n-i}\right)\sum\limits_{\tau \,\vDash n-a}\sum\limits_{b=0}^{k}\sum\limits_{\rho \,\vDash k-b}C_{\tau}C_{a+b}C_{\rho}(1).$$
\end{thm}

\begin{cor}\cite[Corollary 2.2]{sergelparking}
Let $2\leq n\in \mathbb{N}$ and $1\leq k\in \mathbb{N}$, then
$$(-1)^{n-1}\nabla m_{n,1^{k}}=\sum\limits_{P\in \mathcal{PF}_{n+k}}qpoly_{n,k}(P)q^{dinv(P)}t^{area(P)}F_{ides(P)},$$
where $qpoly_{n,k}(P)$ is computed as follows. The constant term is $k+1$. If $P$ is not touching the main diagonal at $(n-1,n-1)$, add $q^{n-1}$. Otherwise stop. If $P$ is not touching the main diagonal at $(n-2,n-2)$, add $q^{n-2}$. Otherwise stop. Continue in this way until you stop or run out of points on the diagonal (this does not include the starting point).
\end{cor}

Our main contribution in this paper is to solve Conjecture \ref{nabla m} in the two-column case. More specifically, we find and prove the following recursion.
\begin{thm}\label{t-2col}
For any $k\in \mathbb{N_+}$ and $l\in \mathbb{N}$,
$$(-1)^{k}m_{2^{k}1^{l}}=\binom{k+l}{k}e_{2k+l}+q\sum\limits_{i=1}^{k}\sum\limits_{j=0}^{l}\binom{i-1+j}{i-1}C_{2i+j}((-1)^{k-i}m_{2^{k-i}1^{l-j}}).$$
\end{thm}
Iterative application of Theorem \ref{t-2col} implies the following result.

\begin{cor}
For any $k\in \mathbb{N_+}$ and $l\in \mathbb{N}$,
  $(-1)^{k} \nabla m_{2^{k}1^{l}}$ is Schur positive.
\end{cor}

In the degenerated case $\mu=(2^k)$, we obtain an explicit combinatorial formula for $(-1)^{k}\nabla m_{2^k}$.

\begin{cor}
For any positive integer $k$,
$$(-1)^{k}\nabla m_{2^k}=\sum_{\substack{P\in \mathcal{PF}_{2k}\\touch(P)=\alpha}} [erun(\alpha)+1]_q q^{dinv(P)}t^{area(P)}F_{ides(P)},$$
where $erun=erun(\alpha)\leq l(\alpha)$ denotes the length of even run of $\alpha$, i.e.,
$\alpha_i$ are even for all $i\leq erun$ and either $l(\alpha)=erun$ or $\alpha_{erun+1}$ is odd.
\end{cor}

One can also formulate a parking function interpretation
for $(-1)^{k} \nabla m_{2^{k}1^{l}} $.

\begin{exa}\label{exa-4.7}
\begin{align*}
m_{2,2,1}&=3e_{5}+q(C_{2}(-m_{2,1})+C_{3}(-m_{2})+C_{4}m_{1}+2C_{5}(1))\\
&=3e_{5}+q(C_{2}(2e_{3}+q(C_{2}e_{1}+C_{3}(1)))+C_{3}(e_{2}+qC_{2}(1))+C_{4}e_{1}+2C_{5}(1))\\
&=3e_{5}+2qC_{2}e_{3}+qC_{3}e_{2}+qC_{4}e_{1}+2qC_{5}(1)+q^{2}C_{2}C_{2}e_{1}+q^{2}C_{2}C_{3}(1)+q^{2}C_{3}C_{2}(1).
\end{align*}
Applying the compositional shuffle theorem to each term gives the desired parking function interpretation. Take the term $qC_{3}e_{2}$
as an instance, we have
\begin{align*}
\nabla (qC_{3}e_{2})&=q \nabla C_{3}e_{2}=q\nabla C_{3}(C_{2}(1)+C_{1}C_{1}(1))\\
&=q\nabla(C_{3,2}+C_{3,1,1})=q\nabla C_{3,2}+q\nabla C_{3,1,1}\\
&=q\cdot\sum_{\substack{P\in \mathcal{PF}_{5}\\touch(P)=(3,2)}}q^{dinv(P)}t^{area(P)}F_{ides(P)}+q\cdot\sum_{\substack{P\in \mathcal{PF}_{5}\\touch(P)=(3,1,1)}}q^{dinv(P)}t^{area(P)}F_{ides(P)}.
\end{align*}
\end{exa}

The paper is organized as follows. In Section 2, we introduce the basic definitions and some fundamental symmetric function identities that will be used for the proofs in the following sections. In section 3, we focus on a degenerate case, i.e. the case where $\mu=(2^k)$. The idea extends for the $\mu=(a^k)$ case using a similar plethystic calculation. We obtain a symmetric function identity related to the Petrie symmetric function introduced by Grinberg in \cite{grinberg2022petrie}. We further propose a conjecture, which will imply the Schur positivity of $(-1)^{(a-1)k}\nabla m_{a^{k}}$ for any positive integer $a$. In Section 4, we prove the main theorem, whose formulation was motivated by a similar result for the $\mu=(2^k)$ case.
It is worth mentioning that the proof techniques in these two sections are completely different.

\section{Definitions}
We follow the notations in Haglund's book \cite{haglund2008q}. The readers can also refer to \cite{bergeron2009algebraic} and \cite{stanley2023enumerative}.

\subsection{Basic Definitions}

\begin{dfn}
The \textbf{q-analogue} of $n\in\mathbb{N}$ is defined by
$$[0]_q =0, \text{ and }\ [n]_{q}=\frac{1-q^n}{1-q}=1+q+\cdots+q^{n-1} \text{ for } n\geq 1.$$
\end{dfn}

\begin{dfn}
A \textbf{partition} $\lambda$ is a nonincreasing finite sequence $\lambda_{1} \geq \lambda_{2} \geq \cdots$ of positive integers. The $\lambda_{i}$ is called the $i$th \textbf{part} of $\lambda$. The number of parts $l(\lambda)$ is called the \textbf{length} of $\lambda$ and the sum of the parts $|\lambda|=\sum_{i}\lambda_{i}$ is called the \textbf{size} of $\lambda$. We define $\lambda \vdash n$ if $\lambda$ is a partition of size $n$.
\end{dfn}

\begin{dfn}
A \textbf{composition} $\alpha$ of $n$, denoted $\alpha \vDash n$, is a positive integer sequence $(\alpha_{1},\alpha_{2},\cdots,\alpha_{r})$ with $|\alpha|=\alpha_{1}+\alpha_{2}+\cdots+\alpha_{r}=n$. The length of $\alpha$ is $l(\alpha)=r$.
\end{dfn}

\begin{dfn}
The \textbf{Young diagram} of $\lambda$ is an array of unit squares, called
cells, with $\lambda_{i}$ cells in the $i$th row (from the bottom), with the first cell in each row left-justified. Define the \textbf{conjugate partition}, $\lambda'$ as the
partition whose Young diagram is obtained from $\lambda$ by reflecting across the diagonal $y=x$.
\end{dfn}

Let $\Lambda$ denote the ring of symmetric functions with coefficients in $\mathbb{Q}(q,t)$.
There are five common bases of $\Lambda$.
\begin{itemize}
\item $\{m_{\lambda}\}$, monomial symmetric functions,
\item $\{e_{\lambda}\}$, elementary symmetric functions,
\item $\{h_{\lambda}\}$, complete homogeneous symmetric functions,
\item $\{p_{\lambda}\}$, power sum symmetric functions,
\item $\{s_{\lambda}\}$, Schur functions.
\end{itemize}

\begin{dfn}
Let $\omega$ denote the \textbf{involution} $\omega: \Lambda\rightarrow\Lambda$ defined by
$$\omega(p_{k})=(-1)^{k-1}p_{k}.$$
In particular, $\omega(e_{\lambda})=h_{\lambda}$, $\omega(h_{\lambda})=e_{\lambda}$ and $\omega(s_{\lambda})=s_{\lambda'}$.
\end{dfn}

\begin{dfn}
The \textbf{Hall scalar product} is a bilinear form from $\Lambda\times \Lambda\rightarrow \mathbb{Q}$, defined by
$$\left<p_{\lambda},p_{\mu}\right>=z_{\lambda}\chi(\lambda=\mu),   \qquad (\chi(true)=1, \ \chi(false)=0)$$
where $z_{\lambda}=\prod\limits_{i}i^{m_{i}}m_{i}!$ with $m_{i}$ being the number of parts of $\lambda$ equal to $i$. In particular, $$\left<h_{\lambda},m_{\mu}\right>=\chi(\lambda=\mu)\quad \text{and}\quad \left<s_{\lambda},s_{\mu}\right>=\chi(\lambda=\mu).$$
\end{dfn}

\begin{dfn}
The \textbf{(modified) Macdonald polynomial} is defined as
$$\tilde{H}_{\mu}=\tilde{H}_{\mu}[X;q,t]=\sum\limits_{\lambda\vdash n}\tilde{K}_{\lambda,\mu}(q,t)s_{\lambda},$$
where $\tilde{K}_{\lambda,\mu}(q,t)$ is the \textbf{(modified) $q,t$-Kostka polynomial}.
\end{dfn}

\begin{dfn}
Let the \textbf{arm} $a = a(x)$, \textbf{leg} $l = l(x)$, \textbf{coarm} $a' = a'(x)$, and
\textbf{coleg} $l'= l'(x)$ be the number of cells strictly between $x$ and the border of $\lambda$ in the
$E$, $N$, $W$ and $S$ directions, respectively (see Figure \ref{f-armleg}). Define
\begin{align*}
n(\mu)=\sum\limits_{x\in \mu}l'(x)=\sum\limits_{x\in \mu}l(x), \quad M=(1-q)(1-t),\\
T_{\mu}=t^{n(\mu)}q^{n(\mu')},\quad w_{\mu}=\prod\limits_{x\in \mu}(q^{a}-t^{l+1})(t^{l}-q^{a+1}).
\end{align*}
\end{dfn}

\begin{dfn}
Let $\nabla$ be the linear operator on symmetric functions which satisfies
$$\nabla \tilde{H}_{\mu}=T_{\mu}\tilde{H}_{\mu}.$$
\end{dfn}

\begin{dfn}
The \textbf{$\ast$-scalar product} is a bilinear form from $\Lambda\times \Lambda\rightarrow \mathbb{Q}$, defined by
$$\left<p_{\lambda},p_{\mu}\right>_{\ast}=(-1)^{|\mu|-l(\mu)}\prod\limits_{i}(1-t^{\mu_{i}})(1-q^{\mu_{i}})z_{\mu}\chi(\lambda=\mu).$$
In particular,
$$\left<\tilde{H}_{\lambda},\tilde{H}_{\mu}\right>_{\ast}=\chi(\lambda=\mu)w_{\mu}.$$
\end{dfn}

\begin{figure}[H]
\centering
\includegraphics[scale=0.9]{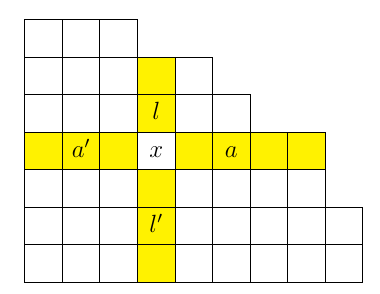}
\caption{The arm, leg, coarm and coleg of a cell.}
\label{f-armleg}
\end{figure}

\begin{dfn}
A \textbf{parking function of size $n$} is a Dyck path of size $n$ with its north steps labeled with the integers 1 to $n$ such that integers in the same column increase (see Figure \ref{parking_function}). The labels are referred to as cars. The set of all parking functions of size $n$ is denoted by $\mathcal{PF}_{n}$.
\end{dfn}

\begin{dfn}
Let $P$ be a parking function of size $n$. The \textbf{area} of $P$ is the same as the area of the underlying Dyck path $\pi(P)$. That is,
$$area(P)=area(\pi(P))=\sum\limits_{i=1}^{n}a_{i}(\pi(P)),$$
where $a_{i}(\pi(P))$ is the number of complete squares between the underlying Dyck path and the main diagonal $y=x$ in the $i$-th row.
\end{dfn}

\begin{dfn}\label{dinv_parkingfunction}
Let $P$ be a parking function of size $n$ with $v_{i}(P)=v_{i}$ the car in the $i$-th row. Define
$$dinv(P)=\sum\limits_{1\leq i<j\leq n}\chi(a_{i}=a_{j} \text{ and } v_{i}<v_{j})+\sum\limits_{1\leq i<j\leq n}\chi(a_{i}=a_{j}+1 \text{ and } v_{i}>v_{j}).$$
\end{dfn}

\begin{dfn}
Let $P$ be a parking function of size $n$, define
$$touch(P)=\alpha=(\alpha_{1},\alpha_{2},\cdots,\alpha_{l(\alpha)})\vDash n,\qquad ret(P)=\alpha_1,$$
if $\pi(P)$ touches the main diagonal $y=x$ in the $(\alpha_1+\cdots+\alpha_i)$-th row for each $i$.
\end{dfn}

\begin{dfn}
For $S\subseteq [n-1]=\{1,\dots, n-1\}$, define \textbf{Gessel's fundamental quasisymmetric function} by
$$F_{n,S}=\sum\limits_{\substack{i_{1}\leq i_{2}\leq \cdots\leq i_{n}\\j\in S\Rightarrow i_{j}<i_{j+1}}}x_{i_{1}}x_{i_{2}}\cdots x_{i_{n}}.$$
When the degree is clear from the context, we will write $F_{S}$ for $F_{n,S}$.
\end{dfn}

\begin{dfn}
Let $P$ be a parking function of size $n$. The \textbf{reading word} $\sigma(P)$ of $P$ is the permutation obtained by reading cars
 along the diagonals, from top to bottom, and from right to left within each diagonal. That is, read the values of $v_{i}(P)$ from large to small values of $a_{i}(P)$ and, within ties, from large to small values of $i$. The \textbf{inverse descent set} of $P$ is
$$ides(P)=Des(\sigma(P)^{-1}).$$
\end{dfn}

\begin{exa}For the parking function depicted in Figure \ref{parking_function}, we have
$area(P)=5$, $touch(P)=(2,4)$, $ret(P)=2$, $$dinv(P)=|\{(1,3)\}|+|\{(2,3),(5,6)\}|=3,$$ $\sigma(P)=524631$, $ides(P)=Des((524631)^{-1})=Des(625314)=\{1,3,4\}$.
\end{exa}

\begin{figure}[htb]
    \centering
    \includegraphics[scale=0.8]{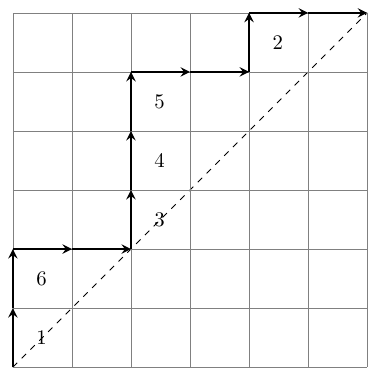}
    \caption{A parking function of size 6.}
    \label{parking_function}
\end{figure}

\subsection{Some Basic Symmetric Functions Identities}

\begin{dfn}
Let $E(t_{1},t_{2},t_{3},\cdots)$ be a formal series of rational functions in the parameters $t_{1},t_{2},t_{3},\cdots$ We define the \textbf{plethystic substitution} of $E$ into $p_{k}$, denoted $p_{k}[E]$, by
$$p_{k}[E]=E(t_{1}^{k},t_{2}^{k},\cdots).$$
\end{dfn}

Readers unfamiliar with plethystic notation can refer to \cite{loehr2011computational} and \cite{zabrocki2015introduction}.
Here we list some identities that will be used frequently in following sections.
\begin{prop}[The addition formulas]
\begin{align*}
e_{n}[X+Y]&=\sum\limits_{i=0}^{n}e_{i}[X]e_{n-i}[Y],\quad e_{n}[X-Y]=\sum\limits_{i=0}^{n}e_{i}[X]e_{n-i}[-Y],\\
h_{n}[X+Y]&=\sum\limits_{i=0}^{n}h_{i}[X]h_{n-i}[Y],\quad h_{n}[X-Y]=\sum\limits_{i=0}^{n}h_{i}[X]h_{n-i}[-Y].
\end{align*}
\end{prop}

\begin{prop}
For any $f\in \Lambda$,
$$f[-\epsilon X]=\omega f[X],$$
where $\epsilon=-1$, but is understood as a monomial whenever it is in the bracket.
\end{prop}

\begin{prop}
Define $\Omega=\sum\limits_{n\geqslant0}h_{n}$, then
\begin{align*}
\Omega[zX]=\sum\limits_{n\geq 0}h_{n}[X]z^{n},\quad
\Omega[X+Y]=\Omega[X]\Omega[Y],\quad \Omega[-X]=\Omega[X]^{-1}.
\end{align*}
\end{prop}

We also need the next simple result.
\begin{prop}
If $u$ and $v$ are monomials, then we have
\begin{equation}\label{p-hn}
h_{n}[(1-u)v]=\begin{cases}
1, & \text{ if } n=0,\\
(1-u)v^{n}, & \text{ if } n\geq 1.
\end{cases}
\end{equation}
\end{prop}

\section{The $\mu=(2^{k})$ case}
Our objective in this section is to establish the following result.
\begin{thm}\label{t-m2k}
For any positive integer $k$,
\begin{equation}\label{expansion}
(-1)^{k}m_{2^k}=e_{2k}+\sum\limits_{i=1}^{k}\sum\limits_{\substack{\left|\alpha\right|\leq k\\l(\alpha)=i}}q^{i}C_{2\alpha_{1}}\cdots C_{2\alpha_{i}}(e_{2k-2\left|\alpha\right|}).
\end{equation}
\end{thm}
By the fact $e_n=\sum_{\alpha \vDash n} C_\alpha$, we obtain a positive $C$-expansion of $(-1)^{k}m_{2^k}$. Thus we have
\begin{cor}
For any positive integer $k$,
$$\left<(-1)^{k}\nabla m_{2^{k}}, s_{\lambda}\right>\in \mathbb{N}[q,t],\quad \forall\ \lambda\vdash 2k.$$
That is, $(-1)^{k}\nabla m_{2^k}$ is Schur positive.
\end{cor}

Besides, we obtain a parking function interpretation for $(-1)^{k}\nabla m_{2^{k}}$ by applying $\nabla$ to both sides of (\ref{expansion}).
\begin{cor}
For any positive integer $k$,
$$(-1)^{k}\nabla m_{2^k}=\sum_{\substack{P\in \mathcal{PF}_{2k}\\touch(P)=\alpha}} [erun(\alpha)+1]_q q^{dinv(P)}t^{area(P)}F_{ides(P)},$$
where $erun=erun(\alpha)\leq l(\alpha)$ denotes the length of the first run of even parts in $\alpha$, i.e.,
$\alpha_i$ are even for all $i\leq erun$ and either $l(\alpha)=erun$ or $\alpha_{erun+1}$ is odd.

\end{cor}
\begin{proof}
By treating $C_0$ as the identity, we can rewrite \eqref{expansion} as follows:
\begin{align*}
 (-1)^{k}m_{2^k}&=\sum\limits_{i=0}^{k}\sum\limits_{\substack{\left|\alpha\right|\leqslant k\\l(\alpha)=i}}q^{i}C_{2\alpha_{1}}\cdots C_{2\alpha_{i}}(e_{2k-2\left|\alpha\right|})\\
 &=\sum\limits_{i=0}^{k}\sum\limits_{\substack{\left|\alpha\right|\leqslant k\\l(\alpha)=i}}q^{i}C_{2\alpha_{1}}\cdots C_{2\alpha_{i}}\sum_{\beta \vDash 2k-2\left|\alpha\right|} C_\beta.
\end{align*}
We now collect terms with respect to $C$.
For each $\gamma=(\gamma_1,\dots, \gamma_l)$, the term $C_\gamma$ arises from the $i$-th summand, where $i=0,1,\dots, erun=erun(\gamma)$, with respect to the compositions $2\alpha=(\gamma_1,\dots, \gamma_i)$ and $\beta=(\gamma_{i+1},\dots, \gamma_l)$.
Therefore, we obtain
$$(-1)^{k}m_{2^k}=\sum_{\alpha \vDash 2k}\left(1+q+\cdots +q^{erun(\alpha)}\right)C_{\alpha}.$$
The corollary is then established by applying the compositional shuffle theorem.
\end{proof}

The formula in \eqref{expansion} is new. We guessed it through investigating some computation data.
The difficulty lies in the fact that $\{C_\alpha\}$ is not a basis so that the $\{C_\alpha\}$ expansion of $(-1)^{k}m_{2^k}$ is not unique.
Once Theorem \ref{t-m2k} is formulated, we are able to transform it into its equivalent form, namely
Proposition \ref{p-m2k} below. On one hand, the theorem follows by iterative application of Proposition \ref{p-m2k}. On
the other hand, we explain as follows.
\begin{proof}[Reduction of Theorem \ref{t-m2k} to Proposition \ref{p-m2k}]
The proof is by induction on $k$.
By grouping the value of $\alpha_{1}$ in the right hand side of \eqref{expansion}, we have
\begin{align*}
RHS
&= e_{2k}+\sum\limits_{i=1}^{k}\sum\limits_{\substack{\left|\alpha\right|\leqslant k\\l(\alpha)=i}}q^{i}C_{2\alpha_{1}}C_{2\alpha_{2}}\cdots C_{2\alpha_{i}}(e_{2k-2\left|\alpha\right|})\\
&=e_{2k}+\sum\limits_{i=1}^{k}\sum\limits_{j=1}^{k}\sum\limits_{\substack{\alpha_{1}=j\\\left|\alpha\right|\leqslant k\\l(\alpha)=i}}q^{i}C_{2j}C_{2\alpha_{2}}\cdots C_{2\alpha_{i}}(e_{2k-2\left|\alpha\right|})\\
&=e_{2k}+\sum\limits_{i=1}^{k}\sum\limits_{j=1}^{k}\sum\limits_{\substack{\left|\beta\right|\leqslant k-j\\l(\beta)=i-1}}q^{i}C_{2j}C_{2\beta_{1}}\cdots C_{2\beta_{i-1}}(e_{2(k-j)-2\left|\beta\right|})\\
(\text{exchange $i$ and $j$})&=e_{2k}+q\sum\limits_{i=1}^{k}C_{2i}\sum\limits_{j=1}^{k}\sum\limits_{\substack{\left|\beta\right|\leqslant k-i\\l(\beta)=j-1}}q^{j-1}C_{2\beta_{1}}\cdots C_{2\beta_{j-1}}(e_{2(k-i)-2\left|\beta\right|})\\
&=e_{2k}+q\sum\limits_{i=1}^{k}C_{2i}\sum\limits_{j=0}^{k-1}\sum\limits_{\substack{\left|\beta\right|\leqslant k-i\\l(\beta)=j}}q^{j}C_{2\beta_{1}}\cdots C_{2\beta_{j}}(e_{2(k-i)-2\left|\beta\right|})\\
&=e_{2k}+q\sum\limits_{i=1}^{k}C_{2i}\left(e_{2(k-i)}+\sum\limits_{j=1}^{k-1}\sum\limits_{\substack{\left|\beta\right|\leqslant k-i\\l(\beta)=j}}q^{j}C_{2\beta_{1}}\cdots C_{2\beta_{j}}(e_{2(k-i)-2\left|\beta\right|})\right)\\
&=e_{2k}+q\sum\limits_{i=1}^{k}C_{2i}\left(e_{2(k-i)}+\sum\limits_{j=1}^{k-i}\sum\limits_{\substack{\left|\beta\right|\leqslant k-i\\l(\beta)=j}}q^{j}C_{2\beta_{1}}\cdots C_{2\beta_{j}}(e_{2(k-i)-2\left|\beta\right|})\right)\\
&=e_{2k}+q\sum\limits_{i=1}^{k}C_{2i}((-1)^{k-i}m_{2^{k-i}}),
\end{align*}
where in the last step, we use \eqref{expansion} with respect to $k-i$ by the induction hypothesis. Thus it suffices to prove Proposition \ref{p-m2k}.
\end{proof}

\begin{prop}\label{p-m2k}
For any positive integer $k$, we have
\begin{equation}\label{recursion}
(-1)^{k}m_{2^k}=e_{2k}+q\sum\limits_{i=1}^{k}C_{2i}((-1)^{k-i}m_{2^{k-i}}).
\end{equation}
\end{prop}
\begin{proof}
Recall that $m_{2^j}=e_j[p_2[X]]=p_2[e_j[X]]$. 
By definition of the $C$ operator, we have
\begin{align*}
C_{2i}m_{2^{k-i}} &= C_{2i}e_{k-i}\left[p_{2}[X]\right]\\
&=\left(-\frac{1}{q}\right)^{2i-1}e_{k-i}\left[p_{2}[X]-\frac{1-\frac{1}{q^2}}{z^2}\right]\Omega[zX]\Bigg|_{z^{2i}}\\
&=\left(-\frac{1}{q}\right)^{2i-1}\sum\limits_{j=0}^{k-i}e_{j}\left[p_{2}[X]\right]e_{k-i-j}\left[-\frac{1-\frac{1}{q^2}}{z^2}\right]\sum\limits_{n\geqslant0}h_{n}z^{n}\Bigg|_{z^{2i}}\\
&=\left(-\frac{1}{q}\right)^{2i-1}\sum\limits_{j=0}^{k-i}m_{2^{j}}(-1)^{k-i-j}h_{k-i-j}\left[\frac{1-\frac{1}{q^2}}{z^2}\right]\sum\limits_{n\geqslant0}h_{n}z^{n}\Bigg|_{z^{2i}}\\
(\text{by Equation (\ref{p-hn})} )&=\left(-\frac{1}{q}\right)^{2i-1}\left[m_{2^{k-i}}h_{2i}+\sum\limits_{j=0}^{k-i-1}(-1)^{k-i-j}\frac{q^2-1}{q^2}m_{2^{j}}h_{2k-2j}\right].
\end{align*}
Plugging this result into the right hand side of \eqref{recursion} gives
\begin{align*}
RHS
&= e_{2k}+q\sum\limits_{i=1}^{k}\left[(-1)^{k-i}\left(-\frac{1}{q}\right)^{2i-1}\left(m_{2^{k-i}}h_{2i}+\sum\limits_{j=0}^{k-i-1}(-1)^{k-i-j}\frac{q^2-1}{q^2}m_{2^{j}}h_{2k-2j}\right)\right]\\
&=e_{2k}+\sum\limits_{i=1}^{k}\left[(-1)^{k-i-1}\frac{q^2}{q^{2i}}m_{2^{k-i}}h_{2i}+\frac{q^2-1}{q^{2i}}\sum\limits_{j=0}^{k-i-1}(-1)^{j-1}m_{2^{j}}h_{2k-2j}\right]\\
&=e_{2k}+\sum\limits_{i=1}^{k}\left[\sum\limits_{j=0}^{k-i}(-1)^{j-1}\frac{1}{q^{2i-2}}m_{2^{j}}h_{2k-2j}-\sum\limits_{j=0}^{k-i-1}(-1)^{j-1}\frac{1}{q^{2i}}m_{2^{j}}h_{2k-2j}\right]\\
&=e_{2k}+\sum\limits_{i=0}^{k-1}\sum\limits_{j=0}^{k-i-1}(-1)^{j-1}\frac{1}{q^{2i}}m_{2^{j}}h_{2k-2j}-\sum\limits_{i=1}^{k}\sum\limits_{j=0}^{k-i-1}(-1)^{j-1}\frac{1}{q^{2i}}m_{2^{j}}h_{2k-2j}\\
&=e_{2k}+\sum\limits_{j=0}^{k-1}(-1)^{j-1}m_{2^{j}}h_{2k-2j}.
\end{align*}
The proposition then follows by showing that 
\begin{equation}\label{e2k}
e_{2k}=\sum\limits_{i=0}^{k}(-1)^{i}m_{2^{i}}h_{2k-2i}.
\end{equation}
This is the special case when $k=2 $ of the following Lemma \ref{l-grinberg2022petrie}.
\end{proof}

\medskip
We are able to extend the idea in Proposition \ref{p-m2k} to the $\mu=(a^k)$ case.

\begin{dfn}
Define Petrie symmetric functions by
$$G(k,n)=\sum\limits_{\lambda\vdash n, \lambda_{1}<k}m_{\lambda}.$$
In particular, $G(2,n)=e_{n}$.
\end{dfn}

We restate a symmetric function identity related to $G(k,n)$ in the following form.
\begin{lem}\cite[Theorem 2.19]{grinberg2022petrie} \label{l-grinberg2022petrie}
$$G(k,n)=\sum_{i\geq 0}(-1)^{i}m_{k^i}h_{n-ki}.$$
\end{lem}

\begin{thm}
For any positive integer $a$ and $k$, we have
$$(-1)^{(a-1)k}m_{a^{k}}=q^{a-1}\sum\limits_{i=1}^{k}C_{ai}\left((-1)^{(a-1)(k-i)}m_{a^{k-i}}\right)+(-1)^{ak}G(a,ak).$$
\end{thm}
\begin{proof}
Using a similar plethystic
calculation as Proposition \ref{p-m2k}, we obtain
\begin{align*}
C_{ai}m_{a^{k-i}} &= C_{ai}e_{k-i}\left[p_{a}[X]\right]\\
&=\left(-\frac{1}{q}\right)^{ai-1}e_{k-i}\left[p_{a}[X]-\frac{1-\frac{1}{q^a}}{z^a}\right]\Omega[zX]\Bigg|_{z^{ai}}\\
&=\left(-\frac{1}{q}\right)^{ai-1}\sum\limits_{j=0}^{k-i}e_{j}\left[p_{a}[X]\right]e_{k-i-j}\left[-\frac{1-\frac{1}{q^a}}{z^a}\right]\sum\limits_{n\geqslant0}h_{n}z^{n}\Bigg|_{z^{ai}}\\
&=\left(-\frac{1}{q}\right)^{ai-1}\sum\limits_{j=0}^{k-i}m_{a^{j}}(-1)^{k-i-j}h_{k-i-j}\left[\frac{1-\frac{1}{q^a}}{z^a}\right]\sum\limits_{n\geqslant0}h_{n}z^{n}\Bigg|_{z^{ai}}\\
&=\left(-\frac{1}{q}\right)^{ai-1}\left[m_{a^{k-i}}h_{ai}+\sum\limits_{j=0}^{k-i-1}(-1)^{k-i-j}\frac{q^a-1}{q^a}m_{a^{j}}h_{ak-aj}\right].
\end{align*}
Then
\begin{align*}
&q^{a-1}\sum\limits_{i=1}^{k}C_{ai}\left((-1)^{(a-1)(k-i)}m_{a^{k-i}}\right)\\
&=q^{a-1}\sum\limits_{i=1}^{k}\left[(-1)^{(a-1)(k-i)}\left(-\frac{1}{q}\right)^{ai-1}\left(m_{a^{k-i}}h_{ai}+\sum\limits_{j=0}^{k-i-1}(-1)^{k-i-j}\frac{q^a-1}{q^a}m_{a^{j}}h_{ak-aj}\right)\right]\\
&=\sum\limits_{i=1}^{k}\left[(-1)^{ak-k+i-1}\frac{q^a}{q^{ai}}m_{a^{k-i}}h_{ai}+\frac{q^a-1}{q^{ai}}\sum\limits_{j=0}^{k-i-1}(-1)^{ak-j-1}m_{a^{j}}h_{ak-aj}\right]\\
&=\sum\limits_{i=1}^{k}\left[\sum\limits_{j=0}^{k-i}(-1)^{ak-j-1}\frac{q^a}{q^{ai}}m_{a^{j}}h_{ak-aj}-\sum\limits_{j=0}^{k-i-1}(-1)^{ak-j-1}\frac{1}{q^{ai}}m_{a^{j}}h_{ak-aj}\right]\\
&=\sum\limits_{i=0}^{k-1}\sum\limits_{j=0}^{k-i-1}(-1)^{ak-j-1}\frac{1}{q^{ai}}m_{a^{j}}h_{ak-aj}-\sum\limits_{i=1}^{k}\sum\limits_{j=0}^{k-i-1}(-1)^{ak-j-1}\frac{1}{q^{ai}}m_{a^{j}}h_{ak-aj}\\
&=\sum\limits_{j=0}^{k-1}(-1)^{ak-j-1}m_{a^{j}}h_{ak-aj},
\end{align*}
and the theorem follows by Lemma \ref{l-grinberg2022petrie} with respect to $G(a,ak)$.
\end{proof}

The Conjecture 5.2 in \cite{grinberg2022petrie} shows that there may be a certain pattern about the Schur positivity of $\nabla G(k,n)$. We focus on a special case and propose the following conjecture.

\begin{cnj}\label{cnj-Gkn}
For any positive integers $a$ and $k$, $(-1)^{ak}\nabla G(a,ak)$ is Schur positive.
\end{cnj}

\begin{cor}
Conjecture \ref{cnj-Gkn} implies the Schur positivity of $(-1)^{(a-1)k}\nabla m_{a^{k}}$.
\end{cor}

\section{Proof of the Main Theorem}
We first introduce the adjoint operator $C_a^{\lor}$ and then prove the main theorem.

\subsection{The adjoint operator $C_a^{\lor}$}
Our proof relies on the adjoint operator $C_a^{\lor}$ of $C_a$ with respect to the Hall scalar product:
$$\left<f, C_{a}g\right>=\left<C_{a}^{\lor}f, g\right>  \text{ for any } f, g\in \Lambda.$$

This is analogous to Garsia's $C_{a}^{\ast}$ operator, which is the adjoint operator of $C_{a}$ under the $\ast$-scalar product:
$$\left<f, C_{a}g\right>=\left<C_{a}^{\ast}f, g\right>_{\ast} \text{ for any } f, g\in \Lambda.$$
The $C_{a}^{\ast}$ operator has a plethystic formula as follows.
\begin{prop}\cite[Theorem 3.6]{garsia2012hall}
For any $f\in \Lambda$ and $a\in\mathbb{Z}$,
$$C_{a}^{\ast}f[X]=\left(-\frac{1}{q}\right)^{a-1}f\left[X-\frac{\epsilon M}{z}\right]\Omega\left[\frac{-\epsilon zX}{q(1-t)}\right]\Bigg|_{\frac{1}{z^{a}}}.$$
\end{prop}

\begin{prop}\cite{garsia2012hall}
For any $f\in \Lambda$, we denote $f^{\ast}=f\left[\frac{X}{M}\right]$, then
$$\left<f,g\right>=\left<f,\omega g^{\ast}\right>_{\ast}=\left<f,g\left[\frac{-\epsilon X}{M}\right]\right>_{\ast}.$$
\end{prop}

We derive an analogous plethystic formula of $C_{a}^{\lor}$ as follows.
\begin{prop}\label{p-calor}
For any $f\in \Lambda$ and $a\in\mathbb{Z}$,
\begin{equation}\label{calor}
C_{a}^{\lor}f[X]=\left(-\frac{1}{q}\right)^{a-1}f\left[X+\frac{1}{z}\right]\Omega\left[\frac{1-q}{q}zX\right]\Bigg |_{\frac{1}{z^{a}}}.
\end{equation}
\end{prop}
We give two proofs of this result.
\begin{proof}[First Proof]
For any $g\in \Lambda$,
\begin{align*}
\left<C_{a}^{\lor}f, g\right>=\left<f, C_{a}g\right>&=\left<f\left[\frac{-\epsilon X}{M}\right], C_{a}g\right>_{\ast}\\
&=\left<C_{a}^{\ast}f\left[\frac{-\epsilon X}{M}\right],g\right>_{\ast}\\
&=\left<\left(-\frac{1}{q}\right)^{a-1}f\left[\frac{-\epsilon}{M}\left(X-\frac{\epsilon M}{z}\right)\right]\Omega\left[\frac{-\epsilon z}{q}\frac{X(1-q)}{M}\right]\Bigg |_{\frac{1}{z^{a}}}, g\right>_{\ast}\\
&=\left<\left(-\frac{1}{q}\right)^{a-1}f\left[\frac{-\epsilon X}{M}+\frac{1}{z}\right]\Omega\left[\frac{-\epsilon}{M}\frac{(1-q)zX}{q}\right]\Bigg|_{\frac{1}{z^{a}}}, g\right>_{\ast}\\
&=\left<\left(-\frac{1}{q}\right)^{a-1}f\left[X+\frac{1}{z}\right]\Omega\left[\frac{1-q}{q}zX\right]\Bigg |_{\frac{1}{z^{a}}}, g\right>.
\end{align*}
\end{proof}

\begin{proof}[Second Proof]
By the bi-linearity of the Hall scalar product, it suffices to assume $f=p_\lambda$ and $g=p_\mu$. Clearly
$\mu\vdash |\lambda|-a$. We need the
following two formulas. One is
$$ p_\lambda[X+Y]= \sum_{p_\alpha p_\beta=p_\lambda}\binom{m(\lambda)}{m(\alpha)}  p_\alpha[X] p_\beta[Y],$$
where $m(\gamma)=(m_1(\gamma),m_2(\gamma),\dots)$ for a partition $\gamma$, and
$$\binom{m(\lambda)}{m(\alpha)}=\prod_{i=1}^{\lambda_{1}}\binom{m_{i}(\lambda)}{m_{i}(\alpha)};$$
the other is
$$ \frac{z_\lambda}{z_\beta} =\frac{\prod_{i} i^{m_{i}(\lambda)}m_{i}(\lambda)!}{\prod_{i} i^{m_{i}(\beta)}m_{i}(\beta)!}=
\prod_{i} i^{m_{i}(\alpha)}m_{i}(\alpha)!\frac{m_{i}(\lambda)!}{m_{i}(\beta)!m_{i}(\alpha)!}
=\binom{m(\lambda)}{m(\alpha)} z_\alpha.$$

We have
\begin{align*}
\left<C_{a}^{\lor}p_{\lambda},p_{\mu}\right>&=\left<p_{\lambda},C_{a}p_{\mu}\right>\\
&=\left<p_{\lambda},\left(-\frac{1}{q}\right)^{a-1}p_{\mu}\left[X-\frac{1-\frac{1}{q}}{z}\right]\Omega\left[zX\right]\Bigg|_{z^{a}}\right>\\
&=\left(-\frac{1}{q}\right)^{a-1}\left<p_{\lambda},\sum\limits_{p_\alpha p_\rho=p_\mu}
\binom{m(\mu)}{m(\alpha)}   p_{\alpha}[X]p_{\rho}
\left[\frac{1-q}{q}\right]\frac{1}{z^{|\rho|}}\sum\limits_{n\geq 0}\sum\limits_{\beta\vdash n}\frac{1}{z_{\beta}}z^{|\beta|}p_{\beta}[X]\Bigg|_{z^{a}}\right>
\\
&=\left(-\frac{1}{q}\right)^{a-1} \sum_{ p_\alpha p_\rho=p_\mu, p_\alpha p_ \beta=p_\lambda}
\binom{m(\mu)}{m(\alpha)} \frac{z^{|\beta|}}{z^{|\rho|}}\frac{z_\lambda}{z_{\beta}}p_{\rho}\left[\frac{1-q}{q}\right] \Bigg|_{z^{a}}\\
&=\left(-\frac{1}{q}\right)^{a-1} \sum_{ p_\alpha p_\rho=p_\mu} \binom{m(\mu)}{m(\alpha)}\binom{m(\lambda)}{m(\alpha)}z_\alpha p_{\rho}\left[\frac{1-q}{q}\right],
\end{align*}
where in the last step, we removed the variable $z$, since $p_\alpha p_\beta=p_\lambda$, $p_\alpha p_\rho=p_\mu$ and $|\lambda|=|\mu|+a$ implies that
$|\beta|-|\rho|=a$.

On the other hand,
\begin{align*}
\left<C_{a}^{\lor}p_{\lambda},p_{\mu}\right>&=\left<\left(-\frac{1}{q}\right)^{a-1}p_{\lambda}\left[X+\frac{1}{z}\right]\Omega\left[\frac{1-q}{q}zX\right]\Bigg|_{\frac{1}{z^{a}}},p_{\mu}\right>\\
&=\left(-\frac{1}{q}\right)^{a-1}\left<\sum\limits_{p_\alpha p_\beta=p_\lambda}
\binom{m(\lambda)}{m(\alpha)}p_{\alpha}[X]p_{\beta}\left[\frac{1}{z}\right]\sum\limits_{n\geq 0}\sum\limits_{\rho\vdash n}\frac{1}{z_{\rho}}p_{\rho}\left[\frac{1-q}{q}\right]z^{|\rho|}p_{\rho}[X]\Bigg|_{\frac{1}{z^{a}}},p_{\mu}\right>\\
&=\left(-\frac{1}{q}\right)^{a-1} \sum_{
p_\alpha p_\beta=p_\lambda, p_\alpha p_\rho =p_\mu}
\binom{m(\lambda)}{m(\alpha)}\frac{z^{|\rho|}}{z^{|\beta|}}
\frac{z_\mu}{z_{\rho}}p_{\rho}\left[\frac{1-q}{q}\right]  \Bigg|_{\frac{1}{z^{a}}}\\
&=\left(-\frac{1}{q}\right)^{a-1} \sum_{
p_\alpha p_\rho =p_\mu}
\binom{m(\lambda)}{m(\alpha)}\binom{m(\mu)}{m(\alpha)}z_{\alpha}p_{\rho}\left[\frac{1-q}{q}\right].
\end{align*}
This completes the proof.
\end{proof}

\begin{rem}
One of the reviewers has presented an additional argument. In the context of the Hall scalar product, the adjoint operation of $f\mapsto f\left[X-\frac{1-\frac{1}{q}}{z}\right]$ is given by $f\mapsto f\Omega\left[\frac{1-q}{qz}X\right]$ and the adjoint operation of $f\mapsto f\Omega[zX]$ is $f\mapsto f[X+z]$. It follows that
$$C_{a}^{\lor}f[X]=\left(-\frac{1}{q}\right)^{a-1}f[X+z]\Omega\left[\frac{1-q}{qz}X\right]\bigg|_{z^a}.$$
This is equivalent to Proposition \ref{p-calor} after using $1/z$ instead of $z$.
\end{rem}

\begin{prop}
For any $a\in\mathbb{Z}$ and partitions $\lambda=(1^{m_{1}}2^{m_{2}}\cdots s^{m_{s}})$ and $\mu$,
\begin{equation}\label{camh}
\left<C_{a}m_{\mu},h_{\lambda}\right>=\left<m_{\mu},\left(-\frac{1}{q}\right)^{a-1}\prod\limits_{i=1}^{s}H_{i}(z)^{m_{i}}\frac{1}{z^{|\lambda|}}\sum\limits_{n\geq 0}\frac{z^{n}}{q^{n}}h_{n}\frac{1}{H(z)}\Bigg|_{\frac{1}{z^{a}}}\right>,
\end{equation}
where $H_{i}(z)=\sum\limits_{j=0}^{i}h_{j}z^{j}$ and $H(z)=\sum\limits_{j=0}^{\infty}h_{j}z^{j}$.
\end{prop}
\begin{proof}
By (\ref{calor}), we calculate directly
\begin{align*}
\left<C_{a}m_{\mu},h_{\lambda}\right>
&=\left<m_{\mu},C_{a}^{\lor}h_{\lambda}\right>\\
&=\left<m_{\mu},\left(-\frac{1}{q}\right)^{a-1}h_{\lambda}\left[X+\frac{1}{z}\right]\Omega\left[\frac{1-q}{q}zX\right]\Bigg |_{\frac{1}{z^{a}}}\right>\\
&=\left<m_{\mu},\left(-\frac{1}{q}\right)^{a-1}\prod\limits_{i=1}^{s}\left(h_{i}\left[X+\frac{1}{z}\right]\right)^{m_{i}}\Omega\left[\frac{1-q}{q}zX\right]\Bigg |_{\frac{1}{z^{a}}}\right>\\
&=\left<m_{\mu},\left(-\frac{1}{q}\right)^{a-1}\prod\limits_{i=1}^{s}\left(\frac{1}{z^{i}}+h_{1}\frac{1}{z^{i-1}}+\cdots+h_{i}\right)^{m_{i}}\Omega\left[\frac{1-q}{q}zX\right]\Bigg |_{\frac{1}{z^{a}}}\right>\\
&=\left<m_{\mu},\left(-\frac{1}{q}\right)^{a-1}\prod\limits_{i=1}^{s}\left(1+h_{1}z+\cdots+h_{i}z^{i}\right)^{m_{i}}\frac{1}{z^{|\lambda|}}\Omega\left[\frac{1-q}{q}zX\right]\Bigg |_{\frac{1}{z^{a}}}\right>\\
&=\left<m_{\mu},\left(-\frac{1}{q}\right)^{a-1}\prod\limits_{i=1}^{s}H_{i}(z)^{m_{i}}\frac{1}{z^{|\lambda|}}\Omega\left[\frac{1-q}{q}zX\right]\Bigg |_{\frac{1}{z^{a}}}\right>,
\end{align*}
besides, we have
$$\Omega\left[\frac{1-q}{q}zX\right]=\Omega\left[\frac{z}{q}X-zX\right]=\Omega\left[\frac{z}{q}X\right]\Omega\left[zX\right]^{-1}=\sum\limits_{n\geq 0}\frac{z^{n}}{q^{n}}h_{n}\frac{1}{H(z)}.$$
This completes the proof.
\end{proof}

\subsection{The proof}
The case $(-1)^{k}m_{2^{k}1^{l}}$ is much more complicated.
A natural approach is to find an analogous result of Theorem \ref{t-m2k}, but such a formula seems too hard to guess according to our computer experiment.
Then we tried to extend Proposition \ref{p-m2k} and succeeded. For the reader's convenience, we restate Theorem \ref{t-2col} as follows.
\begin{thm}\label{main theorem}
For any positive integer $k$ and nonnegative integer $l$,
$$(-1)^{k}m_{2^{k}1^{l}}=\binom{k+l}{k}e_{2k+l}+q\sum\limits_{i=1}^{k}\sum\limits_{j=0}^{l}\binom{i-1+j}{i-1}C_{2i+j}((-1)^{k-i}m_{2^{k-i}1^{l-j}}).$$
\end{thm}
\begin{proof}
It suffices to verify that
\begin{equation} \label{e-main-t}
\left<q\sum\limits_{i=1}^{k}\sum\limits_{j=0}^{l}\binom{i-1+j}{i-1}C_{2i+j}((-1)^{k-i}m_{2^{k-i}1^{l-j}}), h_{\lambda}\right>=
\begin{cases}
(-1)^{k}, & \text{if } \lambda=(2^{k}1^{l}),\\
-\binom{k+l}{k}, & \text{if } \lambda=(1^{2k+l}),\\
0, & \text{otherwise}.
\end{cases}
\end{equation}

For $\lambda\vdash 2k+l$, let
$$u=|\{i: \lambda_{i}=1\}|,\quad v=|\{i: \lambda_{i}\geq 2\}|.$$
Putting  $a=2i+j$ and $\mu=(2^{k-i}1^{l-j})$ in (\ref{camh}), we have
\begin{align*}
&\left<C_{2i+j}((-1)^{k-i}m_{2^{k-i}1^{l-j}}),h_{\lambda}\right>\\
&=\left<(-1)^{k-i}m_{2^{k-i}1^{l-j}}, \left(-\frac{1}{q}\right)^{2i+j-1}H_{1}(z)^{u}H_{2}(z)^{v}\frac{1}{z^{2k+l}}\sum\limits_{n\geq 0}\frac{z^{n}}{q^{n}}h_{n}\frac{1}{H(z)}\Bigg|_{\frac{1}{z^{2i+j}}}\right>\\
&=(-1)^{k-i}\left(-\frac{1}{q}\right)^{2i+j-1}H_{1}(z)^{u}H_{2}(z)^{v}\frac{1}{z^{2k+l}}
\sum\limits_{n\geq 0}\frac{z^{n}}{q^{n}}h_{n}\frac{1}{H(z)}
\Bigg|_{h_{2}^{k-i} h_{1}^{l-j} \frac{1}{z^{2i+j}} h_3^0h_4^0\cdots }\\
&=(-1)^{k-i}\left(-\frac{1}{q}\right)^{2i+j-1}H_{1}(z)^{u}H_{2}(z)^{v}\frac{1}{z^{2k+l}}\left(1+\frac{z}{q}h_{1}+\frac{z^{2}}{q^{2}}h_{2}\right)\frac{1}{H_{2}(z)}
\Bigg|_{h_{2}^{k-i} h_{1}^{l-j} \frac{1}{z^{2i+j}}}\\
&=(-1)^{k-i}\left(-\frac{1}{q}\right)^{2i+j-1}H_{1}(z)^{u}H_{2}(z)^{v-1}\frac{1}{z^{2k+l-2i-j}}
\frac{1}{h_{2}^{k-i}h_{1}^{l-j}}\left(1+\frac{z}{q}h_{1}+\frac{z^{2}}{q^{2}}h_{2}\right)\Bigg|_{h_{2}^{0} h_{1}^{0} z^{0}}.
\end{align*}
In the last step, we use the fact that $\{h_i\}_{ i\ge 1}$ are algebraically independent and generate $ \Lambda$, so that 
we are able to treat $h_i$ as variables and work in the ring of Laurent series in $h_1,h_2$.

Plugging this identity into the left hand side of \eqref{e-main-t}, we obtain
\begin{align*}
LHS
&=q\sum\limits_{i=1}^{k}\sum\limits_{j=0}^{l}\frac{(1+\alpha)^{i-1+j}}{\alpha^{i-1}}(-1)^{k-i}\left(-\frac{1}{q}\right)^{2i+j-1}H_{1}(z)^{u}H_{2}(z)^{v-1}\\
&\quad\quad\quad\quad\quad\quad\quad\quad\quad
\times\frac{1}{z^{2k+l-2i-j}}\frac{1}{h_{2}^{k-i}h_{1}^{l-j}}\left(1+\frac{z}{q}h_{1}+\frac{z^{2}}{q^{2}}h_{2}\right)\Bigg|_{h_{2}^{0}h_{1}^{0}z^{0}\alpha^{0}}\\
&=(-1)^{k}\sum\limits_{i=1}^{\infty}(-1)^{i-1}\left(\frac{z^{2} h_2( 1+\alpha)}{q^{2}\alpha}\right)^{i-1} \sum\limits_{j=0}^{\infty}(-1)^{j}(1+\alpha)^{j}\left(\frac{1}{q}\right)^{j}z^{j}(h_{1})^{j}\\
&\quad\quad\quad\quad\quad\quad\quad\quad\quad
\times H_{1}(z)^{u}H_{2}(z)^{v-1}\frac{z^{2}}{z^{2k+l}}\frac{h_{2}}{h_{2}^{k}h_{1}^{l}}\left(1+\frac{z}{q}h_{1}+\frac{z^{2}}{q^{2}}h_{2}\right)
\Bigg|_{h_{2}^{0}h_{1}^{0}z^{0}\alpha^{0}}\\
&=(-1)^{k}\frac{1}{1+\left(\frac{1}{q^{2}}\frac{1+\alpha}{\alpha}z^{2}h_{2}\right)}\frac{1}{1+\frac{1}{q}(1+\alpha)zh_{1}}(1+h_{1}z)^{u}(1+h_{1}z+h_{2}z^{2})^{v-1}\\
&\quad\quad\quad\quad\quad\quad\quad\quad\quad
\quad\quad\quad\quad\quad\quad
\times\frac{z^{2}}{z^{2k+l}}\frac{h_{2}}{h_{2}^{k}h_{1}^{l}}\left(1+\frac{z}{q}h_{1}+\frac{z^{2}}{q^{2}}h_{2}\right)\Bigg|_{h_{2}^{0}h_{1}^{0}z^{0}\alpha^{0}}
\end{align*}
We eliminate $\alpha$ by the following partial fraction decomposition.
\begin{multline*}
  \frac{1}{1+\left(\frac{1}{q^{2}}\frac{1+\alpha}{\alpha}z^{2}h_{2}\right)}\frac{1}{1+\frac{1}{q}(1+\alpha)zh_{1}}\\
  =
-{\frac{{z}^{2}{h_2}}{\alpha\left(q^{2}+qz h_1+z^{2}h_2\right)}}
\cdot \frac{1}{1+\left(\frac{1}{q^{2}}\frac{1+\alpha}{\alpha}z^{2}h_{2}\right)}
+\frac{q\left(q+zh_1\right) }{q^2+qzh_1+z^{2}h_2}\cdot\frac{1}{1+\frac{1}{q}(1+\alpha)zh_{1}}.
\end{multline*}
By taking the constant term in $\alpha$, the first term vanishes since it contains only negative powers in $\alpha$,
and for the second term, we simply set $\alpha=0$ to obtain
$$\frac{q\left(q+zh_1\right)}{q^2+qzh_1+z^{2}h_2}\cdot \frac{1}{1+\frac{1}{q}zh_{1}}=\frac{q^{2}}{q^{2}+qzh_{1}+z^{2}h_{2}}=\frac{1}{1+\frac{z}{q}h_1+\frac{z^2}{q^2}h_2},$$
since it contains only nonnegative powers in $\alpha$. Thus we have
\begin{align*}
LHS&=(-1)^{k}\frac{1}{1+\frac{z}{q}h_1+\frac{z^2}{q^2}h_2}(1+h_{1}z)^{u}(1+h_{1}z+h_{2}z^{2})^{v-1}\\
&\quad\quad\quad\quad\quad\quad\quad\quad\quad
\quad\quad\quad\quad\quad\quad
\times\frac{1}{z^{2k+l-2}}\frac{1}{h_{2}^{k-1}h_{1}^{l}}\left(1+\frac{z}{q}h_{1}+\frac{z^{2}}{q^{2}}h_{2}\right)\Bigg|_{h_{2}^{0}h_{1}^{0}z^{0}}\\
&=(-1)^{k}(1+h_{1}z)^{u}(1+h_{1}z+h_{2}z^{2})^{v-1}\frac{1}{z^{2k+l-2}}\frac{1}{h_{2}^{k-1}h_{1}^{l}}\Bigg|_{h_{2}^{0}h_{1}^{0}z^{0}}\\
&=(-1)^{k}\sum\limits_{u_{1}=0}^{u}\binom{u}{u_{1}}(h_{1}z)^{u_{1}}\sum\limits_{v_{1}+v_{2}+v_{3}=v-1}
\binom{v-1}{v_{1},v_{2},v_{3}}(h_{1}z)^{v_{2}}(h_{2}z^{2})^{v_{3}}\frac{1}{z^{2k+l-2}}\frac{1}{h_{2}^{k-1}h_{1}^{l}}\Bigg|_{h_{2}^{0}h_{1}^{0}z^{0}}\\
&=(-1)^{k}\sum\limits_{u_{1}=0}^{u}\binom{u}{u_{1}}(h_{1}z)^{u_{1}}\sum\limits_{v_{1}+v_{2}=v-k}\binom{v-1}{k-1}
\binom{v-k}{v_{2}}(h_{1}z)^{v_{2}}\frac{1}{z^{l}}\frac{1}{h_{1}^{l}}\Bigg|_{h_{1}^{0}z^{0}}\\
&=(-1)^{k}\sum\limits_{u_{1}=0}^{u}\binom{u}{u_{1}}\binom{v-1}{k-1}\binom{v-k}{l-u_{1}}\\
&=(-1)^{k}\binom{v-1}{k-1}\binom{v-k+u}{l}.
\end{align*}

When $v=0$, then $u=2k+l$ and
$$LHS=-(-1)^{k-1}\binom{-1}{k-1}\binom{k+l}{l}=-\binom{k-1}{k-1}\binom{k+l}{k}=-\binom{k+l}{k}.$$

When $0<v<k$, then $$\binom{v-1}{k-1}=0\Rightarrow LHS=0.$$

When $v\geq k$, by definition, we have $u+2v\leq 2k+l \iff v-k+u\leq l+k-v$, then $v-k+u\leq l$. If $v-k+u<l$, then $$\binom{v-k+u}{l}=0\Rightarrow LHS=0.$$
Otherwise $v-k+u=l$, $$u+2v\leq 2k+l\Rightarrow v\leq k,$$
then $v=k, u=l$, implying that $\lambda=2^k 1^l$ and
$$LHS=(-1)^{k}\binom{k-1}{k-1}=(-1)^{k}.$$
This completes the proof of Theorem \ref{main theorem}.
\end{proof}

\begin{cor}
For any positive integer $k$ and nonnegative integer $l$,
$$\left<(-1)^{k}\nabla m_{2^{k}1^{l}}, s_{\lambda}\right>\in \mathbb{N}[q,t] \quad \forall\ \lambda\vdash 2k+l.$$
That is, $(-1)^{k}\nabla m_{2^{k}1^{l}}$ is Schur positive.
\end{cor}
\begin{proof}
By induction on $k$, $(-1)^{k} m_{2^{k}1^{l}}$ is $C$-positive.
\end{proof}

\noindent
\textbf{Data availability:} All data generated or analyzed during this study are included in this published article.

\noindent
\textbf{Acknowledgements:} The authors would like to thank Emily Sergel for helpful communications.
The authors would also like to thank the referees for valuable suggestions to improve the representation. This work was partially supported by the National Natural Science Foundation of China [12071311].

\bibliographystyle{plain}
%\bibliography{2-column}

\begin{thebibliography}{10}

\bibitem{bergeron1999identities}
Fran\c{c}ois Bergeron, Adriano~M. Garsia, Mark Haiman, and Glenn Tesler.
\newblock Identities and positivity conjectures for some remarkable operators
  in the theory of symmetric functions.
\newblock {\em Methods and applications of analysis}, 6(3):363--420, 1999.

\bibitem{bergeron2009algebraic}
Fran{\c{c}}ois Bergeron.
\newblock {\em Algebraic Combinatorics and Coinvariant Spaces}.
\newblock CRC Press, 2009.

\bibitem{bergeron2016compositional}
Francois Bergeron, Adriano Garsia, Emily~Sergel Leven, and Guoce Xin.
\newblock Compositional $(km, kn)$-shuffle conjectures.
\newblock {\em International Mathematics Research Notices},
  2016(14):4229--4270, 2016.

\bibitem{bergeron1999science}
Fran{\c{c}}ois Bergeron and Adriano~M. Garsia.
\newblock Science fiction and {Macdonald} polynomials.
\newblock {\em CRM Proceedings and Lecture Notes AMS}, 6:363--429, 1999.

\bibitem{blasiak2024llt}
Jonah Blasiak, Mark Haiman, Jennifer Morse, Anna Pun, and George H. Seelinger.
\newblock LLT polynomials in the Schiffmann algebra.
\newblock {\em Journal f{\"u}r die reine und angewandte Mathematik (Crelles Journal)}, 2024(811):93--133,
  2024.

\bibitem{carlsson2018proof}
Erik Carlsson and Anton Mellit.
\newblock A proof of the shuffle conjecture.
\newblock {\em Journal of the American Mathematical Society}, 31(3):661--697,
  2018.

\bibitem{garsia2002proof}
Adriano~M. Garsia and James Haglund.
\newblock A proof of the $q,t$-{Catalan} positivity conjecture.
\newblock {\em Discrete Mathematics}, 256(3):677--717, 2002.

\bibitem{garsia1993graded}
Adriano~M. Garsia and Mark Haiman.
\newblock A graded representation model for {Macdonald}'s polynomials.
\newblock {\em Proceedings of the National Academy of Sciences},
  90(8):3607--3610, 1993.

\bibitem{garsia2012hall}
Adriano~M. Garsia, Guoce Xin, and Mike Zabrocki.
\newblock {Hall-Littlewood} operators in the theory of parking functions and
  diagonal harmonics.
\newblock {\em International Mathematics Research Notices}, 2012(6):1264--1299,
  2012.

\bibitem{garsia2014three}
Adriano~M. Garsia, Guoce Xin, and Mike Zabrocki.
\newblock A three shuffle case of the compositional parking function
  conjecture.
\newblock {\em Journal of Combinatorial Theory, Series A}, 123(1):202--238,
  2014.

\bibitem{grinberg2022petrie}
Darij Grinberg.
\newblock Petrie symmetric functions.
\newblock {\em Algebraic Combinatorics}, 5(5):947--1013, 2022.

\bibitem{grojnowski2007affine}
Ian Grojnowski and Mark Haiman.
\newblock Affine {Hecke} algebras and positivity of {LLT} and {Macdonald}
  polynomials.
\newblock {\em Unpublished manuscript}, 2007.

\bibitem{haglund2004proof}
James Haglund.
\newblock A proof of the $q,t$-{Schr{\"o}der} conjecture.
\newblock {\em International Mathematics Research Notices}, 2004(11):525--560,
  2004.

\bibitem{haglund2008q}
James Haglund.
\newblock {\em The $q,t$-{Catalan} numbers and the space of diagonal harmonics:
  with an appendix on the combinatorics of {Macdonald} polynomials}, volume~41.
\newblock American Mathematical Society, 2008.

\bibitem{haglund2005combinatorial}
James Haglund, Mark Haiman, Nicholas Loehr, Jeffrey~B. Remmel, and Alexander
  Ulyanov.
\newblock A combinatorial formula for the character of the diagonal
  coinvariants.
\newblock {\em Duke Mathematical Journal}, 126(2):195--232, 2005.

\bibitem{haglund2012compositional}
James Haglund, Jennifer Morse, and Mike Zabrocki.
\newblock A compositional shuffle conjecture specifying touch points of the
  {Dyck} path.
\newblock {\em Canadian Journal of Mathematics}, 64(4):822--844, 2012.

\bibitem{haglund2018delta}
James Haglund, Jeffrey~B. Remmel, and Andrew Wilson.
\newblock The delta conjecture.
\newblock {\em Transactions of the American Mathematical Society},
  370(6):4029--4057, 2018.

\bibitem{haiman2002vanishing}
Mark Haiman.
\newblock Vanishing theorems and character formulas for the {Hilbert} scheme of
  points in the plane.
\newblock {\em Inventiones mathematicae}, 149:371--407, 2002.

\bibitem{lascoux1997ribbon}
Alain Lascoux, Bernard Leclerc, and Jean-Yves Thibon.
\newblock Ribbon tableaux, {Hall--Littlewood} functions, quantum affine
  algebras, and unipotent varieties.
\newblock {\em Journal of Mathematical Physics}, 38(2):1041--1068, 1997.

\bibitem{loehr2011computational}
Nicholas Loehr and Jeffrey~B. Remmel.
\newblock A computational and combinatorial expos{\'e} of plethystic calculus.
\newblock {\em Journal of Algebraic Combinatorics}, 33(2):163--198, 2011.

\bibitem{loehr2007square}
Nicholas Loehr and Gregory Warrington.
\newblock Square $q,t$-lattice paths and $\nabla p_{n}$.
\newblock {\em Transactions American Mathematical Society}, 359(2):649, 2007.

\bibitem{loehr2008nested}
Nicholas Loehr and Gregory Warrington.
\newblock Nested quantum {Dyck} paths and $\nabla s_{\lambda}$.
\newblock {\em International Mathematics Research Notices}, 2008, 2008.

\bibitem{sergel2016combinatorics}
Emily Sergel.
\newblock {\em The combinatorics of nabla $p_{n}$ and connections to the
  rational shuffle conjecture}.
\newblock University of California, San Diego, 2016.

\bibitem{sergel2017proof}
Emily Sergel.
\newblock A proof of the square paths conjecture.
\newblock {\em Journal of Combinatorial Theory, Series A}, 152:363--379, 2017.

\bibitem{sergelparking}
Emily Sergel.
\newblock A parking function interpretation for $\nabla m_{n,1^{k}}$.
\newblock {\em Séminaire Lotharingien de Combinatoire}, 80B(69), 2018.

\bibitem{sergel}
Emily Sergel.
\newblock A combinatorial model for $\nabla m_{\mu}$.
\newblock {\em arXiv preprint}, 1804.06037, 2018.

\bibitem{stanley2023enumerative}
Richard Stanley.
\newblock {\em Enumerative Combinatorics: Volume 2}.
\newblock Cambridge University Press, 2023.

\bibitem{van2020shuffle}
Stephanie van Willigenburg.
\newblock The shuffle conjecture.
\newblock {\em Bulletin of the American Mathematical Society}, 57(1):77--89,
  2020.

\bibitem{zabrocki2015introduction}
Mike Zabrocki.
\newblock {\em Introduction to symmetric functions}.
\newblock \url{http://garsia. math.yorku.ca/ghana03/mainfile.pdf}, 2015.



\end{thebibliography}

\end{document}